\documentclass{amsart}
\usepackage{amssymb,latexsym}
\usepackage{amsmath}
\usepackage{graphicx}
\usepackage{enumerate}
\usepackage{newlattice}
\usepackage{gensymb}

\theoremstyle{plain}
\newtheorem{theorem}{Theorem}
\newtheorem{lemma}[theorem]{Lemma}
\newtheorem{corollary}[theorem]{Corollary}

\theoremstyle{definition}
\newtheorem{definition}[theorem]{Definition}

\newcommand{\seven}{\mathsf S_7}
\newcommand{\Chl}[1]{\textup{C}_{\tup{l}}(#1)}
\newcommand{\Chr}[1]{\textup{C}_{\tup{r}}(#1)}
\newcommand{\cornl}[1]{\textup{lc(#1)}}
\newcommand{\cornr}[1]{\textup{rc(#1)}}
\newcommand{\Rank}[1]{\textup{Rank(#1)}}
 
\begin{document}
\title{On slim rectangular lattices}
\author[G.\ Gr\"atzer]{George Gr\"atzer}
\email{gratzer@me.com}
\urladdr{http://server.maths.umanitoba.ca/homepages/gratzer/}
\address{University of Manitoba}
\date{\today}

\begin{abstract} 
Let $L$ be a slim, planar, semimodular lattice
(slim means that it does not contain an ${\mathsf M}_3$-sublattice).
We call the interval $I = [o, i]$ of $L$ \emph{rectangular},
if there are complementary $a, b \in  I$ such that $a$ is to the left of $b$.

We claim that a rectangular interval of a slim rectangular lattice 
is also a slim rectangular lattice.
We will present some applications, 
including a recent result of G. Cz\'edli.

In a paper with E. Knapp about a dozen years ago, 
we introduced \emph{natural diagrams} for slim rectangular lattices.
Five years later, G. Cz\'edli introduced \emph{${\E C}_1$-diagrams}.
We prove that they are the same.
\end{abstract}

\maketitle

\section{Introduction}\label{S:recIntroduction}
%Section~\ref{S:recIntroduction}

In 2006, we started studying planar semimodular lattices
in my papers with E.~Knapp  \cite{GKn07}--\cite{GKn10}.
More than four dozen publications have been devoted to this topic
since; see G. Cz\'edli's list\\
\verb+http://www.math.u-szeged.hu/~czedli/m/listak/publ-psml.pdf+

\smallskip

An \emph{SPS lattice} $L$ is a planar semimodular lattice that is also \emph{slim}
(it does not contain an $\SM 3$-sublattice).

Following my paper with E.~Knapp~\cite{GKn09},
a planar semimodular lattice $L$ is \emph{rectangular},
if its left boundary chain has exactly one doubly-irreducible element 
other than the bounds (the \emph{left corner})
and its right boundary chain
has exactly one doubly-irreducible element 
other than the bounds (the \emph{right corner})
and the two corners are complementary.
An \emph{SR lattice} $L$ is a rectangular lattice that is also \emph{slim}.

Rectangular lattices are easier to work with than planar semimodular lattices, 
because they have much more structure. 
Moreover, a planar semimodular lattice has~a (congruence-preserving) extension
to a rectangular lattice, so we can prove many result for SPS
lattices by verifying them for SR lattices
(G.~Gr\"atzer and E.~Knapp~\cite{GKn09}).

It turns out that there is another way to obtain SR lattices
from SPS lattices.
Before we state it, we need a definition. 
Let $L$ be a planar lattice.
We call the interval $I = [o, i]$ of $L$ \emph{rectangular},
if there are complementary $a, b \in  I$ 
such that the element $a$ is to the left of the element $b$.

Now we state a new property of SR lattices.

\begin{theorem}\label{T:Main}
Let $L$ be an slim, planar, semimodular lattice 
and let $I$ be a rectangular interval of $L$.
Then the lattice $I$ is slim and rectangular.
\end{theorem}

In a paper with E. Knapp about a dozen years ago, 
we introduced \emph{natural diagrams} for SR lattices.
Five years later, G. Cz\'edli introduced \emph{${\E C}_1$-diagrams}.
We prove that they are the same.

We will present some applications, including a recent result of G. Cz\'edli \cite{gCcc}.

For the background of this topic and its applications outside lattice theory, 
see Section 1.2 of G. Cz\'edli and G. Gr\"atzer~\cite{CG21}.

\subsection*{Statements and declarations}\hfill

\emph{Data availability statement.} 
Data sharing is not applicable to this article 
as no datasets were generated or analysed during the current study.

\emph{Competing interests.} Not applicable as there are no interests to report.

\subsection*{Basic concepts and notation.}

The basic concepts and notation not defined in this note 
are freely available in Part~I of the book \cite{CFL2}, see \\
{\tt arXiv:2104.06539}\\
We will reference it as CFL2.

\section{Fork extensions}

We discuss in Section 4.3 of CFL2
a result of G.~Cz\'edli and E.\,T. Schmidt~\cite{CS13}:
for an SPS lattice $L$ and covering square $C$ in $L$,
we can \emph{insert} a fork in $L$ at $C$ 
to obtain the lattice extension $L[C]$,
which is also an SPS lattice, see Figure~\ref{F:fork}.
In this figure, the elements of the covering square $C$ 
are grey filled, the elements of the fork are black filled.
The third and fourth diagrams represent the same lattice,
\emph{De~gustibus non est disputandum}.

\begin{figure}[htb]%Figure~\ref{F:fork}
\centerline{\includegraphics[scale=0.60]{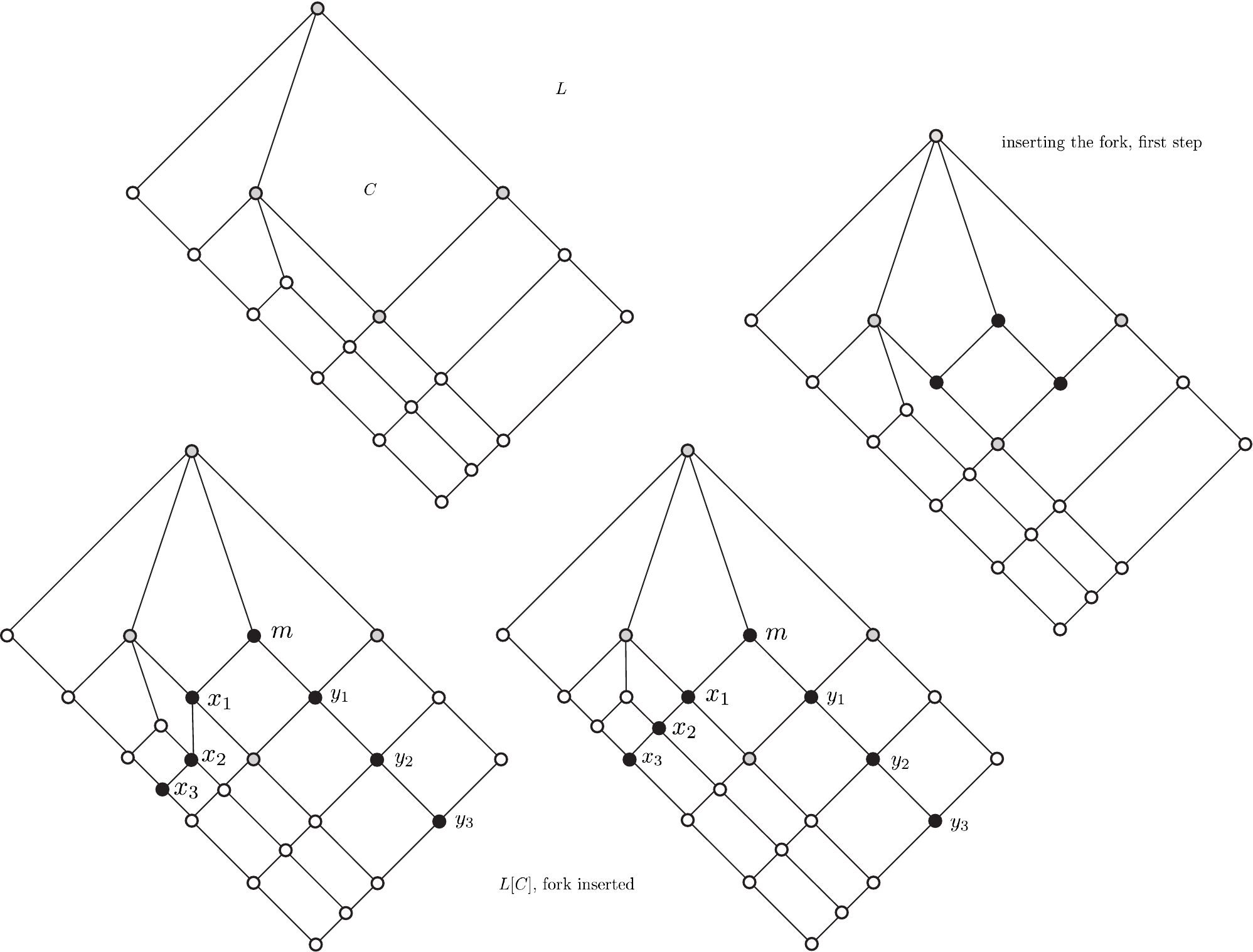}}
\caption{Inserting a fork into $L$ at $C$}\label{F:fork}
\end{figure}

As illustrated by Figure~\ref{F:fork-delete}, 
we can sometimes \emph{delete} a fork.
Let $L$  be an SPS lattice and let $S$ be a covering $\seven$ in $L$,
with middle element $m$, left corner $a$ and right corner $b$.
Let us assume that the top element~$t$ of $S$ is \emph{minimal},
that is, there is no $S'$ a covering $\seven$ 
with top element~$t'$ that is smaller: that is, $t' < t$.

\begin{figure}[t!]%Figure~\ref{F:fork-delete}
\centerline{\includegraphics[scale=.8]{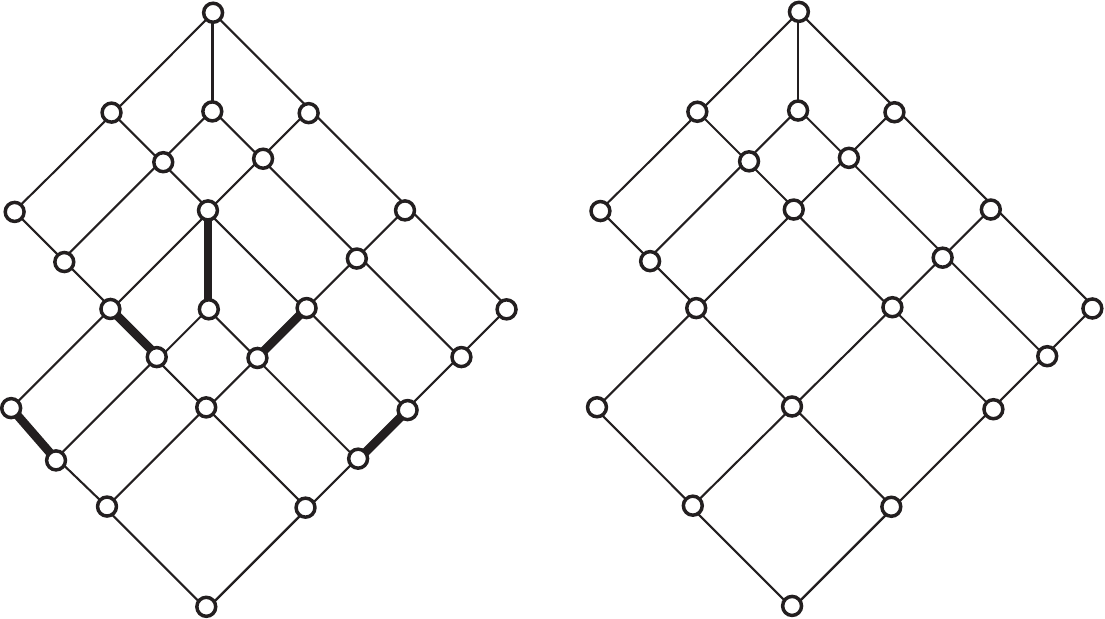}}
\caption{Deleting a fork}\label{F:fork-delete}
\end{figure}

\begin{lemma}[G. Cz\'edli and E.\,T. Schmidt \cite{CS13}]\label{L:delete}
%Lemma~\ref{L:delete}
Let $L$  be an SR lattice and let \[S = \set{o, m \mm a, m \mm b, a,b,m,t}\] 
be a minimal cover\-ing~$\seven$ in $L$.
Then~$L$ has a sublattice~$L^-$ with a covering square 
\[
   C = S - \set{m, m \mm a, m \mm b} = \set{o, a,b,t}
\]
such that $L = L^-[C]$. In other words, we can delete the fork in $S$
and the lattice~$L^-$ is the lattice $L$ with the fork deleted.
\end{lemma}

The structure of SR lattices is described as follows, 
see G.~Cz\'edli and E.\,T. Schmidt~\cite{CS13}.

\begin{theorem}[Structure Theorem]\label{T:structureshort}
%Theorem~\ref{T:structureshort}
A slim rectangular lattice $K$ can be obtained 
from a grid $G$ by inserting forks \lp$n$-times\rp.
\end{theorem}

We thus associate a natural number $n$ with an SR lattice~$K$; 
we call it the \emph{rank} of $K$, and denote it by $\Rank K$. 
It is easy to see that the $\Rank K$ is well defined. 

There is a stronger version of Theorem~\ref{T:structureshort},
implicit in G.~Cz\'edli and E.\,T.~Schmidt~\cite{CS13}.
We~present it with a short proof.

\begin{theorem}[Structure Theorem, Strong Version]\label{T:Structure}
%Theorem~\ref{T:Structure}
For every slim rectangular lattice~$K$, 
there is a grid~$G$, a natural number $n = \Rank K$,
and sequences 
\begin{equation}\label{E:latticesequence}%\eqref{E:latticesequence}
G = K_1, K_2, \dots, K_{n-1}, K_n = K
\end{equation}
of slim rectangular lattices and 
\begin{equation}\label{E:Csequence}%\eqref{E:Csequence}
C_1 = \set{o_1, c_1,d_1, i_1}, C_2  = \set{o_2, c_2,d_2, i_2} , 
      \dots, C_{n-1} = \set{o_{n-1}, c_{n-1},d_{n-1}, i_{n-1}}
\end{equation}
of $4$-cells in the appropriate lattices such that 
\begin{equation}\label{E:Clatticesequence}%\eqref{E:Clatticesequence}
  G = K_1, K_1[C_1] = K_2, \dots, K_{n-1}[C_{n-1}] = K_n= K
\end{equation}
and the principal ideals $\id{c_{n-1}}$ and $\id{d_{n-1}}$ are distributive.
\end{theorem}

\begin{proof}
We prove by induction on $n$.
If $n = 0$, then $K$ is distributive by G.~Gr\"atzer and E. Knapp \cite{GKn09}, 
so the statement is trivial.
Now let us assume that the statement holds for $n-1$. 
Let $K$ be an~SR lattice with $n$ covering $\seven$-s.
As in Lemma~\ref{L:delete}, we take $S$, a~\emph{minimal} covering $\seven$ in $K$.
Then we form the  sublattice~$K^-$ by deleting the fork at $S$.
So~we get a covering square $C = C_{n-1} = \set{o_{n-1}, c_{n-1},d_{n-1}, i_{n-1}}$
of~$K^-$ such that $K = K^-[C]$. 
Since $K^-$ has $n-1$ covering $\seven$-s, we get the sequence
\[
   G = K_1, K_1[C_1] = K_2, \dots, K_{n-2}[C_{n-2}] = K_{n-1} = K^-,
\]
which, along with $K = K^-[C]$, prove the statement for $K$.
The minimality of $S$ implies that 
the principal ideals $\id{c_{n-1}}$ and~$\id{d_{n-1}}$ are distributive.
\end{proof}

\section{Proving Theorem~\ref{T:Main}}\label{S:Proving}
%Section~\ref{S:Proving}

Theorem~\ref{T:Main} obviously holds for grids.

Otherwise, we can assume that the SR lattice $K$ is not a grid,
so  $n = \Rank K>1$. 
Let~$K^-$ be the lattice we obtain 
by deleting a minimal fork in~$K^-$ at the covering square
 \[
    C_{n-1} = \set{o_{n-1}, c_{n-1},d_{n-1}, i_{n-1}}.
 \]
We obtain $K$ from $K^-$ by inserting a fork at $C_{n-1}$.
We add the element $m$ in the middle of $C_{n-1}$, 
and add the sequences of elements $x_1, \dots$ 
on the left going down and $y_1, \dots$ 
on the right going down as in Figure~\ref{F:fork}.

Let  $I$ be a rectangular interval in $K$ with corners $a, b$, 
where $a$ is to the left of $b$.
We want to prove that $I$ is an SR lattice.
Of~course, the lattice~$I$ is slim.

We induct on $n = \Rank K$. 
There are three subcases.

Case 1. $I$ is disjoint to $\id m$, as illustrated in Figure~\ref{F:forkcase1}.
Then the interval $I$ is not changed as we add the fork to $K^-$. 
By induction, $I$ is rectangular in $K^-$, therefore,~$I$ is also rectangular in $K$.

Case 2.  In Figure~\ref{F:forkcase2} (and Figure~\ref{F:forkcase3}), 
the bold lines form the boundary of the rectangular sublattice $I$ in $K^-$,
the elements of $C_{n-1}$ are grey filled,
and the elements~$m$, $x_1$, \dots,  $y_1$, \dots\ are black filled.
The element $m$ is internal in $I$,
so the element $a$ is $c_{n-1}$ or it is to the left of~$c_{n-1}$
and symmetrically, see Figure~\ref{F:forkcase2}.
Therefore, $C_{n-1} = [o_{n-1},i_{n-1}]_{K^-}$ is a covering square in ${K^-}$
and we obtain the interval $[o_{n-1},i_{n-1}]_{K}$ of $K$
by adding a fork to $C_{n-1}$ at  $[o_{n-1},i_{n-1}]_{K^-}$.
A fork extension of an SR lattice is also an SR lattice,
so we conclude that $I$ is an SR lattice.

Case 3.  $m$ is not an internal element of $I$ 
but some $x_i$ or $y_i$ is, see Figure~\ref{F:forkcase3},
where~$y_2$ is an internal element of $I$.
By utilizing that $\id{d_{n-1}}$ is distributive,
we conclude that we obtain  $I$ from  $[o,i]_{K^-}$ 
by replacing a cover preserving $\SC m \times \SC 2$
by $\SC m \times \SC 3$, 
and so $I$ remains rectangular.

\section{Applications of Theorem~\ref{T:Main}}\label{S:Applications}
%Section~\ref{S:Applications}

The next statement follows directly from Theorem~\ref{T:Main}.

\begin{corollary}\label{C:main}
%Corollary~\ref{X:main}
Let $L$ be an SPS lattice and let $I$ be a rectangular interval of~$L$.
Let \lp P\rp be any property of SR lattices. 
Then the property \lp P\rp holds for the lattice $I$.
\end{corollary}

 \newpage
 
%\begin{figure}[p]
\centerline{\includegraphics[scale=.9]{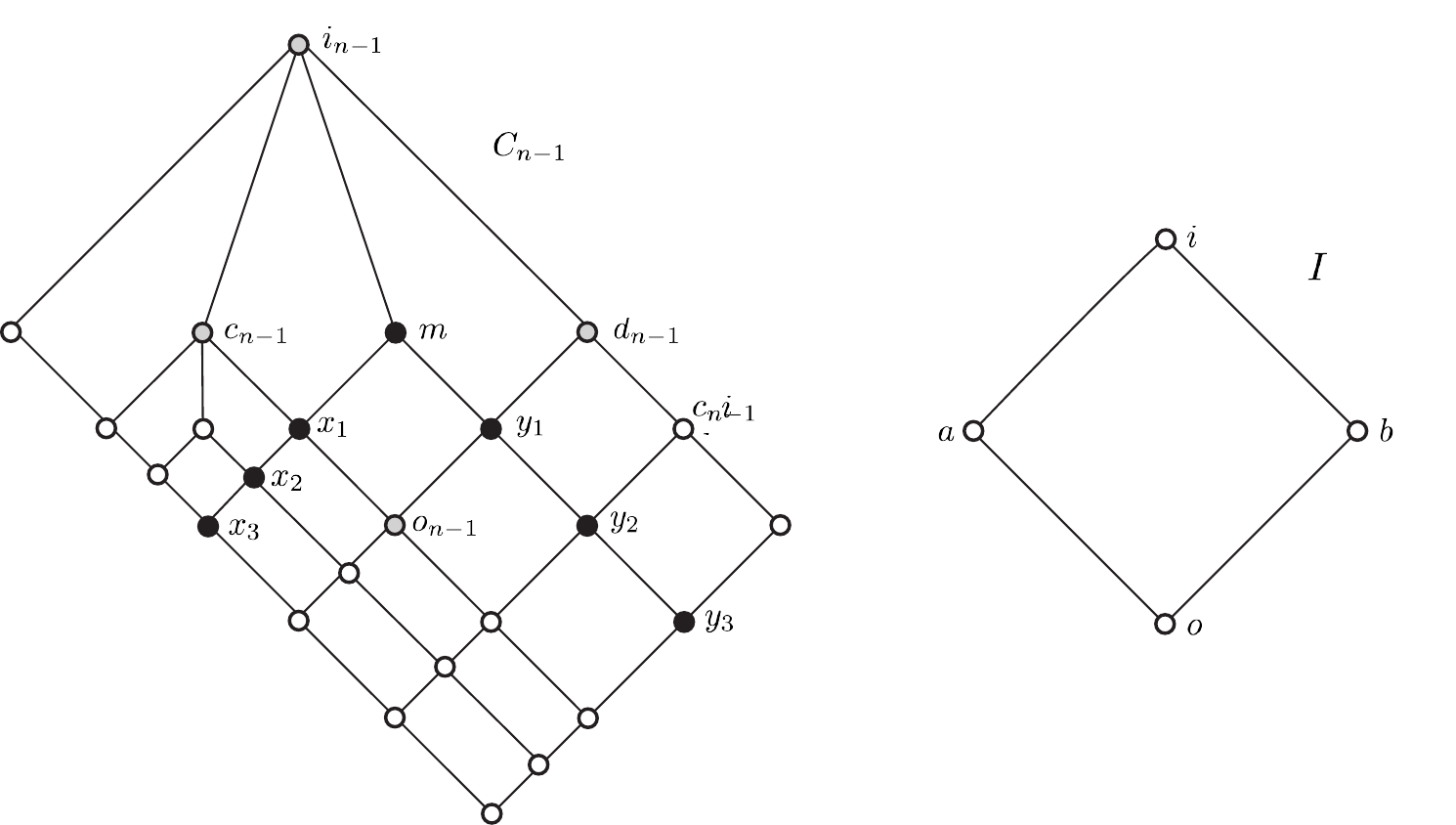}}
%\caption{forkcase1}
%\label{F:forkcase1}

\begin{figure}[htb]
%\centerline{\includegraphics[scale=.65]{}}
\caption{Proving Theorem~\ref{T:Main}: Case 1}
\label{F:forkcase1}
%Figure~\ref{F:forkcase1}
\end{figure}

\bigskip

\bigskip

\bigskip

\bigskip

\centerline{\includegraphics[scale=.9]{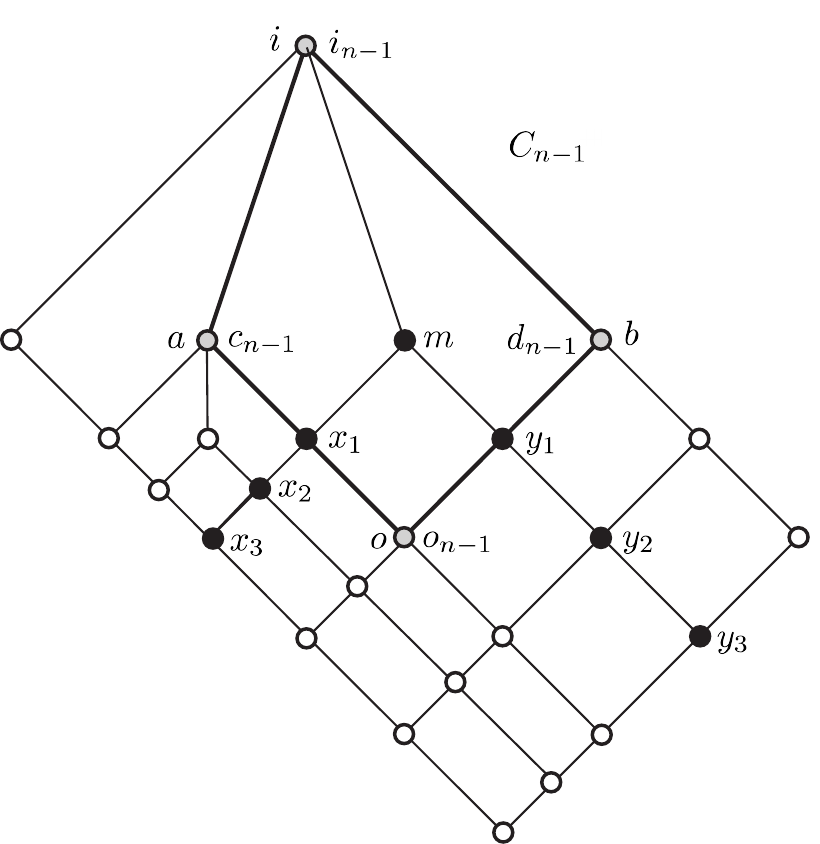}}
%\caption{forkcase2}
%\label{F:forkcase2}
%Figure~\ref{F:forkcase2}
%\end{figure}
%

\begin{figure}[htb]
%\centerline{\includegraphics[scale=.65]{}}
\caption{Proving Theorem~\ref{T:Main}: Case 2}
\label{F:forkcase2}
%Figure~\ref{F:forkcase2}
\end{figure}

\newpage

\begin{figure}[htb]
\centerline{\includegraphics[scale=1]{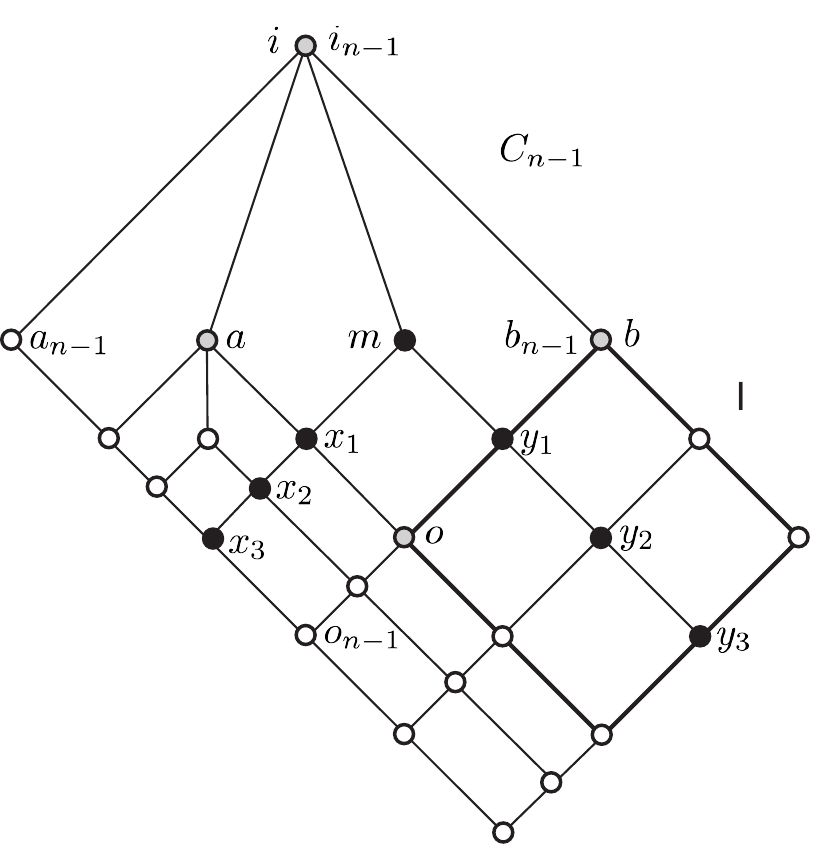}}
\caption{Proving Theorem~\ref{T:Main}: Case 3}
\label{F:forkcase3}
%Figure~\ref{F:forkcase3}
\end{figure}

\begin{figure}[h!]
\centerline{\includegraphics[scale = 1]{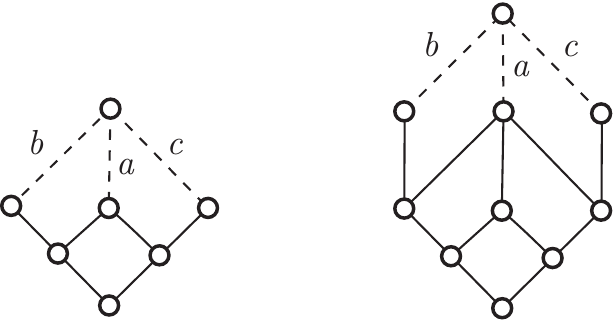}}
\caption{The lattice $\seven$, two diagrams}
\label{F:S7+}
%Figure~\ref{F:S7+}
\end{figure}

For instance, let (P) be the property:
the intervals $[o, a]$ and $[o, b]$ are chains and 
all elements of the lower boundary of $I$ are meet-reducible, except for $a, b$.
Then we get the main result of G. Cz\'edli \cite{gCcc}.

\begin{corollary}\label{C:rectint}%Corollary~\ref{C:rectint}
Let $L$ be an SPS lattice and let $I$ be a rectangular interval of $L$
with corners $a, b$.
Then $[o, a]$ and $[o, b]$ are chains and 
all the elements of the lower boundary of $I$ except for~$a, b$  are meet-reducible.
\end{corollary}

Another nice application is the following.

\begin{corollary}\label{C:meet}%Corollary~\ref{C:meet}
Let $L$ be an SPS lattice and let $I$ be a rectangular interval of $L$
with corners $a, b$.
Then for any $x \in I$, the following equation holds:
\[
   x = (x \mm a) \jj (x \mm b).
\]
\end{corollary}

Here is a more elegant way to formulate the last result.

\begin{corollary}\label{C:abc}%Corollary~\ref{C:abc}
Let $L$ be an SPS lattice and let $a,b,c$ be 
pairwise incomparable elements of $L$. 
If $a$ is to the left of $b$, and $b$ is to the left of $c$, then 
\[
   b = (b \mm a) \jj (b \mm c).
\]
\end{corollary} 

\section{Special diagrams}

\subsection{Natural diagrams}

SR lattices have some particularly nice diagrams
such as the \emph{natural diagrams} of my paper with E. Knapp~\cite{GKn10},
which laid the foundation of rectangular lattices.
Natural diagrams were discovered more than a~dozen years ago, many years before the appearance of it competitor, the $\E C_1$-diagrams of  G.~Cz\'edli---see the next section.

For an SR lattice $L$, 
let $\Chl L$ be the lower left and $\Chr L$ the lower right boundary chain of $L$, respectively, and let $\cornl L$ be the left and
$\cornr L$ the right corner of $L$, respectively.

We regard $G = \Chl L \times \Chr L$ as a planar lattice, 
with $\Chl L = \Chl G$ and $\Chr L = \Chr G$. 
Then the map
\begin{equation}\label{E:gy}%\eqref{E:gy}
\gy \colon  x \mapsto (x \mm \cornl L, x \mm \cornr L)
\end{equation}
is a meet-embedding of $L$ into $G$;
the map $\gy$ also preserves the bounds.
Therefore, the image of $L$ under $\gy$ in $G$ is a diagram of $L$, 
we call it the \emph{natural diagram} representing $L$. 
For~instance, the second diagram of Figure~\ref{F:S7+} 
shows the natural diagram representing~ $\seven$.

Let $L$ be an SR lattice.
By the Structure Theorem, Strong Version,
we can represent~$L$ in the form $L = K[C]$, 
where $K$ is an SR lattice, 
$C = \set{o, c, d, i}$ is a $4$-cell of~$K$ satisfying that $\id c$ and  $\id d$ are distributive.
Let $\E D$ be a diagram of $K$.
We~form the diagram $\E D[C]$ 
by adding the elements $m, x_1, \dots$, and  $m, y_1, \dots$, 
as in the last diagram of Figure~\ref{F:fork}, so that the lines spanned by
the elements $m, x_1, \dots$ and m, $y_1, \dots$ are both normal.

\begin{lemma}\label{L:C1}
%Lemma~\ref{L:C1}
Let $L$, $C$, $K$, $\E D$, and $\E D[C]$ be as in the previous paragraph.
Then $\E D[C]$ is a diagram of~$L$.
\end{lemma}

\begin{proof}
This is obvious.
\end{proof}

\begin{lemma}\label{L:C2}
%Lemma~\ref{L:C2}
Let us make the assumptions of Lemma~\ref{L:C2}.
If $\E D$ is a natural diagram of $K$, 
then $\E D[C]$ is a natural diagram of~$L$.
\end{lemma}

\begin{proof}
So let $\E D$ be a natural diagram of $K$.
Let the line $m, x_1, \dots$ terminate with~$x_{k_l}$
and the line $m, y_1, \dots$ with $y_{k_r}$.
We have to show that all the new elements in $L$
can be represented as a join $u_l \jj u_r$, 
where $u_l \in \Chl L$ and  $u_r \in \Chr L$.
Indeed, $m = x_{k_l} \jj x_{k_r}$. 
The others follow from the distributivity assumptions.
\end{proof}

\subsection*{$\E C_1$-diagrams}

This research tool, introduced by G. Cz\'edli,
has been playing an important role in some recent papers,
see G. Cz\'edli \cite{gC17}--\cite{gCcc},
G. Cz\'edli and G.~Gr\"atzer~\cite{CG21}, and G.~Gr\"atzer~\cite{gG21};
for the definition, see G.~Cz\'edli \cite{gC17} and G.~Gr\"atzer~\cite{gG21}.

In the diagram of an SR lattice $K$,
a \emph{normal edge} (\emph{line}) 
has a slope of $45\degree$ or~$135\degree$.
Any edge (line) of slope 
strictly between $45\degree$ and $135\degree$ is \emph{steep}.

Figure~\ref{F:S7+} depicts the lattice  $\seven$.
A \emph{peak sublattice}~$\seven$ 
(\emph{peak sublattice}, for short) of a lattice $L$ is a sublattice 
isomorphic to  $\seven$ such that the three edges at the top 
are covers in the lattice $L$.

\begin{definition}
%Definition~\ref{D:well}
A diagram of a slim rectangular $L$ is a \emph{${\E C}_1$-diagram},
if the middle edge of a peak sublattice
is steep and all other edges are normal.
\end{definition}

In other words, an edge is steep if it is the middle edge of a peak sublattice;
if an edge is not the middle edge of a peak sublattice, then it is normal.

\begin{theorem}\label{T:well}
%Theorem~\ref{T:well}
Every slim rectangular lattice $L$ has a ${\E C}_1$-diagram.
\end{theorem}

This was proved in G. Cz\'edli \cite{gC17}.
My note \cite{gG21a} presents a short and direct proof.

\section{Natural diagrams and  ${\E C}_1$-diagrams are the same}
\label{S:natural = $C_1$}
%Section~\ref{S:natural = $C_1$}

We start with a trivial statement.

\begin{lemma}\label{L:C2}
%Lemma~\ref{L:C2}
Let us make the assumptions of Lemma~\ref{L:C2}.
If $\E D$ is a ${\E C}_1$-diagram of $K$, 
then $\E D[C]$ is a ${\E C}_1$-diagram of~$L$.
\end{lemma}

Now we state our second result on SR lattices.

\begin{theorem}\label{T:C1=natural}
%Theorem~\ref{T:C1=natural}
Let $L$ be a SR lattice. 
Then a natural diagram of $L$ is a ${\E C}_1$-diagram. 
Conversely, every ${\E C}_1$-diagram is natural.
\end{theorem}

\begin{proof}
Let us assume that the SR lattice~$L$  
can be obtained from a~grid~$G$ by adding forks  $n$-times, where $n = \Rank L$.
We induct on $n$. The case $n = 0$ is trivial because then $L$ is a grid.
So let us assume that the theorem holds for $n - 1$.

By the Structure Theorem, Strong Version,
there is a SR lattice~$K$ 
and a $4$-cell $C = \set{o,a,b,i}$ of~$K$ satisfying that $\id c$ and $\id d$
are distributive 
such that $K$ can be obtained from the grid~$G$ by adding forks $(n-1)$-times
and also $L = K[C]$ holds. 

Now form the natural diagram $\E D$ of $K$. 
By induction, it is a ${\E C}_1$-diagram.
By~Lemma~\ref{L:C1}, the diagram $\E D[C]$ is both natural and ${\E C}_1$. 

We prove the converse the same way.
\end{proof}

Natural diagrams exist by definition. 
So Theorem~\ref{T:well} also follows from Theorem~\ref{T:C1=natural}.

G. Czédli \cite{gC17} also defined \emph{${\E C}_2$-diagrams}.
A ${\E C}_1$-diagram is ${\E C}_2$,
if any two edges on the lower boundary are of the same length.

We use Theorem~\ref{T:C1=natural} 
to prove two results of G. Cz\'edli \cite{gC17}.

\begin{theorem}\label{T:C-2}
%Lemma~\ref{T:C-2}
Let $L$ be a SR lattice. Then $L$ has a ${\E C}_2$-diagram.
\end{theorem}

\begin{proof}
Let $C_l$ and $C_r$ be chains of the same length as 
$\Chl L$ and $\Chr L$, respectively.
Then $\Chl L \times \Chr L$ and $C_l \times C_r$ are isomorphic,
so we can regard the map $\gy$, see~\eqref{E:gy}, 
as a map from $L$ into $C_l \times C_r$,
a bounded and meet-preserving map.
So the natural diagram it defines is the diagram of the lattice $L$.

If we choose $C_l$ and $C_r$ so that the edges are of the same size, 
we obtain a ${\E C}_2$-diagram of the SR lattice $L$.
\end{proof}

Natural diagrams have a left-right symmetry. 
The symmetric diagram is obtained with the map
\begin{equation}\label{E:gy2}%\eqref{E:gy2}
\widetilde{\gy} \colon  x \mapsto (x \mm \cornr L, x \mm \cornl L)
\end{equation}
replacing \eqref{E:gy}.

\begin{theorem}[Uniqueness Theorem]\label{T:Uniqueness}
%Theorem~\ref{T:Uniqueness}
Let $L$ be a slim rectangular lattice.
Then the ${\E C}_2$-diagram of $L$ is unique up to left-right symmetry.
\end{theorem}

\end{document}